\documentclass{amsart}
\usepackage{graphicx}

\newtheorem{theorem}{Theorem}[section]
\newtheorem{lemma}[theorem]{Lemma}

\theoremstyle{definition}

\theoremstyle{remark}

\numberwithin{equation}{section}

\begin{document}

\title{On Bloch's Theorem and the Contraction Mapping Principle}

\author{Jean C. Cortissoz \and Julio A. Montero}
\address{Department of Mathematics, Universidad de los Andes, Bogot\'a DC, COLOMBIA.
}
%\curraddr{Department of Mathematics and Statistics,
%Case Western Reserve University, Cleveland, Ohio 43403}
%\email{xyz@math.university.edu}

%\author{Author Two}
%\address{Mathematical Research Section, School of Mathematical Sciences,
%Australian National University, Canberra ACT 2601, Australia}

%    General info
%\subjclass[2000]{Primary 54C40, 14E20; Secondary 46E25, 20C20}

%\date{January 1, 2001 and, in revised form, June 22, 2001.}


%\keywords{Differential geometry, algebraic geometry}

\begin{abstract}
In this paper we give a new proof, relying on Banach's contraction mapping principle,
of a celebrated theorem of Andr\'e Bloch. Also, via the same contraction mapping principle,
we give a proof of a Bloch type theorem for normalised Wu $K$-mappings.
\end{abstract}

\maketitle
\section{Introduction}

In 1924 Andr\'e Bloch proved his celebrated theorem, whose statement we give right below.
\begin{theorem}
Let $f\left(z\right)$ be an analytic function on $B_1\left(0\right)$ (the unit disk centered at the origin)
satisfying $f\left(0\right)=0$ and $f'\left(0\right)=1$. Then there is a constant $B$ (called 
Bloch's constant) independent of $f$, such that there is a subdomain
$\Omega\subset B_1\left(0\right)$ where $f$ is one to one and whose
image contains a disk of radius $B$ (which we shall call, as is customary, a {\bf schlicht disk}, 
and we will say that $f$ covers a \bf{schlicht disk of radius} B).
\end{theorem}

Bloch used this theorem to give a prove of Picard's Theorem. Since then computing the value of $B$
has been one of the most important problems in geometric
complex analysis. A simple proof 
is given in \cite{Conway} which gives $B\geq \dfrac{1}{72}$. The best estimate from below for Bloch's constant so far belongs to
Xiong \cite{Xiong} ($\dfrac{\sqrt{3}}{4}+3\times 10^{-4}\sim 0.4333\dots$), whereas an estimate from above
was given by Ahlfors and Grunsky in \cite{Ahlfors}; it is conjectured that the value of
Bloch constant is the one given by Ahlfors and Grunsky upper bound ($\sim 0.4719$).

In this paper we present a simple proof of Bloch's theorem based
on the contraction mapping principle (Banach's fixed point theorem). Although we do not
 improve on the known estimates given for Bloch's constant in the
 one dimensional setting, the
 proof seem simple enough to be interesting, and it applies with very few modifications to
 the problem of proving Bloch's theorem analogues in several complex variables. To justify 
 this last claim 
 we will prove a version of Bloch's theorem for the family of (normalised) Wu $K$-mappings,
 with an estimate for Bloch's constant
 that is asymptotically better in $K$ than the results known before Chen and
 Gauthier's paper of 2001 (see \cite{Chen2} and
 the references therein).

This paper is organised as follows. In Section \ref{prooffixedpoint} we give a proof
of Bloch's theorem using the contraction mapping principle. In Section
\ref{UsingEarleHamilton}, using Landau's simplification and the 
Earle-Hamilton fixed point theorem to further refine our estimates
(which, incidentally, gives another proof of Bloch's theorem which, though less elementary,
also gives a very decent estimate from below for Bloch's constant);
in Section \ref{Pushingthemethod}, we suggest using the methods of Section \ref{UsingEarleHamilton},
via an example, an algorithm to estimate Bloch's
constant for families of polynomials, and which may
lead to finding better estimates for Bloch's constant than the estimates currently known. Finally, in Section \ref{Severalvariables}
we apply the ideas developed in Section \ref{prooffixedpoint} to higher dimensions, namely
to the case of Wu $K$-mappings.

\section{A proof of Bloch's Theorem using Banach's Fixed Point Theorem}
\label{prooffixedpoint}

As we said in the introduction, our proof is based on Banach's contraction mapping principle;
the other ingredients in the proof are Cauchy's estimates and the
Maximum Principle. The following well known lemma gives a condition which ensures that a holomorphic map
is a contraction; we give a proof for the sake of completeness (we must say that 
the first version of this lemma that we wrote was incorrect: we thank professor Dror Varolin for
pointing this out and for correcting our wrong statement).

\begin{lemma}
\label{contractionlemma}
Let $f:\Omega\longrightarrow \mathbb{C}$ be a holomorphic function, $\Omega$ a convex domain.
If $\left|f'\left(z\right)\right|\leq 1-\sigma$, $\sigma>0$, for
all $z\in \Omega$, then $f$ is a contraction, i.e.,
\[
\left|f\left(z_2\right)-f\left(z_1\right)\right|\leq \left(1-\sigma\right)\left|z_2-z_1\right|.
\]
\end{lemma}
\begin{proof}
By the Cauchy-Riemann equations, $df$, the derivative of $f$, seen as a linear transformation
from $\mathbb{R}^2$ to $\mathbb{R}^2$ is the composition
of a rotation and a homothety (a reflection is discarded as holomorphic
functions preserve orientation) whose scaling factor is precisely
given by $\left|f'\left(z\right)\right|\leq 1-\sigma$. 
It follows then that
$df$ shortens the length  of vectors by a factor of $1-\sigma$. Therefore, if we let $z_1,z_2\in \Omega$ and
consider the segment joining these two points
\[
\gamma\left(t\right)=\left(1-t\right)z_1+tz_2, \quad 0\leq t\leq 1,
\]
we have that (here $\left\|\cdot\right\|$ denotes the euclidean norm)
\begin{eqnarray*}
\left|f\left(z_2\right)-f\left(z_1\right)\right|&\leq& \int_0^{1}
\left\|df\left(\gamma'\left(t\right)\right)\right\|\,dt\\
&\leq& \left(1-\sigma\right)\int_0^1\left\|\gamma'\left(t\right)\right\|\,dt=
\left(1-\sigma\right)\left|z_2-z_1\right|
\end{eqnarray*}
which is what we wanted to prove.
\end{proof}
So let $f:B_1\left(0\right)\longrightarrow \mathbb{C}$ be a holomorphic function
such that $f\left(0\right)=0$ and $f'\left(0\right)=1$, which without
loss of generality we may assume bounded. We let $\beta\in B_1\left(0\right)$ (from here onwards, the ball of radius $\rho$ centered
at $a$ will be denoted by $B_{\rho}\left(a\right)$), and expand $f$ around $\beta$:
\[
f\left(z\right)= f\left(\beta\right)+
\sum_{n\geq 1}\frac{f^{\left(n\right)}\left(\beta\right)}{n!}
\left(z-\beta\right)^{n}.
\]
We choose $\beta$ such that
\[
\left|f'\left(\beta\right)\right|=
\max_{\left|z\right|=\left|\beta\right|}\left|f'\left(z\right)\right|,
\]
and adopt the notation
\[
M\left(\beta\right)=\max_{\left|z\right|=\left|\beta\right|}\left|f'\left(z\right)\right|.
\]
Fix $\beta_0=\beta_{\gamma}$ and let $\gamma$ be such that 
\begin{equation}
\label{definitionforbeta}
\left|\beta_0\right|\prod_{j\geq 1}\left(1+\gamma^{-j}\right)= 1.
\end{equation}
Construct the finite, maximal, sequence
$\epsilon_j$, $j=0,1,2,\dots,m-1$, $m\geq 1$, $\epsilon_0=0$, with the property that 
\[
\frac{M\left(\beta_{n+1}\right)}{M\left(\beta_n\right)}=\gamma,
\quad \left|\beta_{n}\right|=\prod_{j=1}^{n}\left(1+\epsilon_j\right)\left|\beta_{0}\right|.
\]
To this sequence, we add $\epsilon_{m}\geq 0$ such that
\begin{equation}
\label{definitionepsilon}
\left|\beta_0\right|\prod_{j=1}^{m}\left(1+\epsilon_{j}\right)= 1.
\end{equation}
Notice then that $\beta_{m}=1$ and that $M\left(\beta_{m}\right)/M\left(\beta_{m-1}\right)\leq \gamma$.
Also, notice that at least for one $k$, it must hold that $\epsilon_k\geq \gamma^{-k}$.

The fixed point problem we want to solve is the following
\begin{equation}
\label{fixedpoint1}
z=\frac{w-f\left(\beta_n\right)}{f'\left(\beta_n\right)}+\beta_n
-\sum_{k\geq 2}\frac{f^{\left(k\right)}\left(\beta_n\right)}
{k!f'\left(\beta_n\right)}
\left(z-\beta_n\right)^{k}=:g_{w}\left(z\right).
\end{equation}
The main idea now (and of the whole argument) is to show that there is a disk, say $D$,
such that if $w\in D$, we can restrict $g_w$ to a ball
centered at $\beta_n$, say $D'$, whose radius is independent of $w$, so that 
$g_w:D'\longrightarrow D'$, and so that it is a contraction. Then Banach's
contraction mapping principle tells us that (\ref{fixedpoint1}) has a unique solution
and we can conclude that the radius of $D$ is a bound from below for Bloch's constant.

Therefore, as we are going to make use of Lemma \ref{contractionlemma}, we need the left-hand side of the previous identity 
to satisfy the inequality
\begin{equation}
\label{contractioncondition}
\left|g'_w\left(z\right)\right|\leq\sum_{k\geq 2}\left|
\frac{f^{\left(k\right)}\left(\beta_n\right)}
{\left(k-1\right)!f'\left(\beta_n\right)}
\right|\left|z-\beta_n\right|^{k-1}<1.
\end{equation}
We must then find $\eta>0$ such that for all $z\in B_{\eta}\left(\beta_n\right)$
the previous inequality holds, and also 
we will also need this $\eta$ to be such that
\[
g_{w}:B_{\eta}\left(\beta_n\right)
\longrightarrow B_{\eta}\left(\beta_n\right),
\]
which will be fulfilled if
\begin{equation}
\label{contractioncondition2}
\left|w-f\left(\beta_n\right)\right|
\leq \eta M\left(\beta_n\right)
-\sum_{k\geq 2} \frac{\left|f^{\left(k\right)}\left(\beta_n\right)\right|}{k!}\eta^k.
\end{equation}
The right hand side of (\ref{contractioncondition2}) will give an estimate for the radius of
the disk $D$ mentioned above.

Let us work a bit on inequality (\ref{contractioncondition}).
By Cauchy's estimates we know that
\[
\left|f^{\left(k\right)}\left(\beta_n\right)\right|
\leq \frac{M\left(\beta_{n+1}\right)\left(k-1\right)!}
{\left(\epsilon_{n+1}\left|\beta_n\right|\right)^{k-1}},
\]
so we can change the contraction condition (\ref{contractioncondition}) by the more stringent
\[
\frac{M\left(\beta_{n+1}\right)}{M\left(\beta_n\right)}
\sum_{k\geq 2}\left(\frac{\eta}{\epsilon_{n+1}\left|\beta_n\right|}\right)^{k-1}
\leq 1-\sigma,
\]
which is equivalent to the inequality
\[
\eta\leq \frac{1-\sigma}{1-\sigma+\gamma}\epsilon_{n+1}\left|\beta_n\right|.
\]
Let us now concentrate on inequality (\ref{contractioncondition2}). First notice that 
\begin{eqnarray*}
\sum_{k\geq 2} \frac{\left|f^{(k)}\left(\beta_n\right)\right|}{k!}\eta^k
&=&M\left(\beta_n\right)\sum_{k\geq 2} 
\frac{\left|f^{(k)}\left(\beta_n\right)\right|}
{k!\cdot M\left(\beta_n\right)}\eta^k\\
&\leq&
M\left(\beta_n\right)\sum_{k\geq 2}
\frac{\left(k-1\right)!M\left(\beta_{n+1}\right)}{k! M\left(\beta_n\right)}
\frac{\eta^k}{\left(\epsilon_{n+1}\beta_n\right)^{k-1}}\\
&\leq&
M\left(\beta_n\right)\frac{1}{2}\eta
\left(1-\sigma\right).
\end{eqnarray*}
We then can change (\ref{contractioncondition2}) by the stronger inequality
\begin{eqnarray*}
\left|w-f\left(\beta_n\right)\right|&\leq& \eta M\left(\beta_n\right)-
\frac{1-\sigma}{2}\eta M\left(\beta_n\right)\\
&=&\left(\frac{1}{2}+\frac{\sigma}{2}\right)\eta M\left(\beta_n\right).
\end{eqnarray*}
This shows that Bloch's constant is bigger than the optimum of
\begin{equation}
\label{blochestimate1}
\left(\frac{\sigma}{2}+\frac{1}{2}\right)\frac{1-\sigma}{1-\sigma+\gamma}\epsilon_{n+1}\left|\beta_n\right| M\left(\beta_n\right),
\quad n=0,1,2,\dots, m-1.
\end{equation}

To estimate (\ref{blochestimate1}) from we below, we estimate $\left|\beta_{\gamma}\right|$, as we clearly have that
$\left|\beta_{\gamma}\right|\leq \left|\beta_n\right|$.
Starting from (\ref{definitionforbeta}), we obtain
\[
\log \left|\beta_{\gamma}\right|=-\sum_{j\geq 1}\log\left(1+\gamma^{-j}\right),
\]
and then by Taylor's theorem
\[
\log\left(1+\gamma^{-j}\right)\leq \gamma^{-j}-\frac{1}{2}\gamma^{-2j}+\frac{1}{3}\gamma^{-3j},
\]
so we have
\[
\log \left|\beta_{\gamma}\right|\geq -\frac{1}{\gamma-1}+\frac{1}{2}\frac{1}{\gamma^2-1}-\frac{1}{3}\frac{1}{\gamma^3-1},
\]
that is,
\[
\left|\beta_{\gamma}\right|\geq e^{-\frac{1}{\gamma-1}+\frac{1}{2}\frac{1}{\gamma^2-1}-\frac{1}{3}\frac{1}{\gamma^3-1}}.
\]
Observe that there must be a $n\leq m-1$ for which
$\epsilon_{n+1}M\left(\beta_n\right)\geq \dfrac{M\left(\beta_{\gamma}\right)}{\gamma}$, this because of the definition
of the $\beta_k$'s and the observation right after (\ref{definitionepsilon}).
Hence, $f$ covers a schlicht disk of radius at least
\begin{equation}
\label{Blochfrombelow}
\left(\frac{\sigma}{2}+\frac{1}{2}\right)
\frac{1-\sigma}{1-\sigma+\gamma}\frac{1}{\gamma}
e^{-\frac{1}{\gamma-1}+\frac{1}{2}\frac{1}{\gamma^2-1}-\frac{1}{3}\frac{1}{\gamma^3-1}}M\left(\beta_{\gamma}\right).
\end{equation}
The previous inequality already proves Bloch's theorem, since by the
Maximum Principle $M\left(\beta_{\gamma}\right)\geq 1$. Maximising the
expression above we obtain $B\geq 0.0355493>1/29$.

However, we can do a bit better.
Indeed, under the assumption that $M\left(\beta_{\gamma}\right)>\gamma$, (\ref{Blochfrombelow})
implies that 
$f$ would cover a schlicht disk of radius
\[
\left(\frac{1}{2}+\frac{\sigma}{2}\right)\frac{1-\sigma}{1-\sigma+\gamma}
e^{-\frac{1}{\gamma-1}+\frac{1}{2}\frac{1}{\gamma^2-1}-\frac{1}{3}\frac{1}{\gamma^3-1}}.
\]
On the other hand, if we apply the previous reasoning, as we did to equation (\ref{fixedpoint1}), under the assumption
that $M\left(\beta_{\gamma}\right)\leq \gamma$ to the equation
\begin{equation}
\label{contractioncondition3}
z=w-\sum_{k=2}^{\infty}\frac{f^{\left(k\right)}\left(0\right)}{k!}z^k,
\end{equation}
we have that for the right-hand side to be a contraction we need that
\[
\sum_{k=2}^{\infty}\frac{\left|f^{\left(k\right)}\left(0\right)\right|}{\left(k-1\right)!}
\left|z\right|^{k-1}\leq 1-\sigma,
\]
which, by using Cauchy's estimates, becomes 
\[
\sum_{k=2}^{\infty}M\left(\beta_{\gamma}\right)
\left|\frac{z}{\beta_{\gamma}}\right|^{k-1}\leq 1-\sigma,
\]
which gives the restriction
\[
\left|z\right|\leq \frac{1-\sigma}{1-\sigma+M\left(\beta_{\gamma}\right)}\left|\beta_{\gamma}\right|.
\]
Also, we need, for the function defined
by the right-hand side of (\ref{contractioncondition3}) to map a disk 
of a given radius centered at 0 into the same disk, that
\[
\left|w\right|\leq \left(\frac{1}{2}+\frac{\sigma}{2}\right)
\frac{1-\sigma}{1-\sigma+M\left(\beta_{\gamma}\right)}\left|\beta_{\gamma}\right|.
\]
Therefore, if $M\left(\beta_{\gamma}\right)\leq \gamma$ (and hence, by the
discussion above in any case),
we find that $f$ has a schlicht disk of radius at least 
\[
\left(\frac{1}{2}+\frac{\sigma}{2}\right)\frac{1-\sigma}{1-\sigma+\gamma}
e^{-\frac{1}{\gamma-1}+\frac{1}{2}\frac{1}{\gamma^2-1}-\frac{1}{3}\frac{1}{\gamma^3-1}}.
\]
Hence, if we optimise the previous function we obtain
the best lower bound for Bloch's constant that this method can provide.
Using Wolfram we have obtained an estimate from below of
approximately $0.0813782$ which is larger than $1/13$. 

%%%%%%%%%%%%%%%%%%%%%%%%%%%%%%%%%%%%%%%%%%%%%%%%%%%%%%%%

\section{Refining the method: Landau's simplification and the Earle-Hamilton 
Fixed Point Theorem}
\label{UsingEarleHamilton}

We shall show a small refinement that allows an interesting improvement
 the estimates given so far for
Bloch's constant. 
But first, we need to introduce an new tool, the Earle-Hamilton
fixed point theorem, and a simplification of the problem due to E. Landau.
\subsection{The Earle-Hamilton Fixed Point Theorem}
 
 Let us recall the statement of this beautiful theorem.
 
 \begin{theorem}
 Let $D\subset \mathbb{C}$ be bounded, and let $f:D\longrightarrow D$ be a holomorphic function
 such that the distance between
 $f\left(D\right)$ and $\mathbb{C}\setminus D$ is bigger or equal than a fixed positive constant.
 Then $f$ has a unique fixed point.
 \end{theorem}
 
 This theorem allows us to get rid of the constraint generated by the inequality 
 deduced from the fact that we need $g_w$ to be a contraction. Hence, we just
 need to work with the constraint that tells us that $g_w$ goes
 from a ball centered at the origin of certain radius to a ball also centered at
 the origin of a slightly smaller radius (as small as desired, so that we can assume 
 in our estimates that it has the same radius).
 
 \subsection{Landau's simplification}
 Landau in \cite{Landau} showed that to obtain Bloch's constant, one can look
for it among the functions that satisfy
\[
\left|f'\left(z\right)\right|\leq \frac{1}{1-\left|z\right|^2}, 
\quad f'\left(0\right)=1,\quad f\left(0\right)=0.
\]
We say that a function satisfying the above set of constraints satisfies
Landau's simplification or that it is a Bloch function.

Landau's simplification implies that the Taylor series of $f$ can be written as
\[
f\left(z\right)=z+\frac{1}{3}a_2 z^3+\sum_{k\geq 4}\frac{f^{(k)}\left(0\right)}{k!}z^k,
\quad \mbox{with}\quad \left|a_2\right|\leq 1.
\]
 
\subsection{The estimate for Bloch's constant}
Given $f$ a Bloch function, as defined above, we can write its derivative as
\[
f'\left(z\right)=1+a_2z^2+a_3z^3+\dots.
\]
Hence we can estimate, for the coefficients of this expansion, with
$n\geq 4$, that for any $0<r<1$ (see the proof of Lemma 1 in \cite{Bonk})
\[
\left|a_n\right|r^{2n}\leq \frac{1}{\left(1-r^2\right)^2}-1-\left|a_2\right|^2r^4-\left|a_3\right|^2r^6.
\]
Thus,
\begin{equation}
\label{Bonkcoeffineq2}
\left|a_n\right|\leq \frac{1}{r^{n-1}}\sqrt{\frac{2-r^2}{\left(1-r^2\right)^2}-\left|a_2\right|^2r^2-\left|a_3\right|^2r^4},
\quad 0<r<1.
\end{equation}
The fixed point problem we must solve is
\begin{equation}
\label{EarleHamFPPeq}
z=w-\sum_{k\geq 2}\frac{a_k}{k+1}z^{k+1}=:g_w\left(z\right).
\end{equation}
According
to the Earle-Hamilton fixed point theorem,
equation (\ref{EarleHamFPPeq}) will have a unique solution as long as, for a
given $\rho$, $w$ satisfies the following estimate for every $z\in B_{\rho}\left(0\right)$
\begin{equation}
\label{mappingineq}
\left|w\right|+\frac{\left|a_2\right|}{3}\left|z\right|^3+
\frac{\left|a_3\right|}{4}\left|z\right|^4+\sum_{k\geq 4}\frac{\left|a_k\right|}{k+1}\left|z\right|^{k+1}
\leq \rho.
\end{equation}

We fix $\rho$ as 0.45,
which means that we have fixed a ball of radius 0.45 centered at the origin in the
$z$-plane. Now we must 
show that $g_w$ sends this ball into itself.
This will be so, using (\ref{mappingineq}) and (\ref{Bonkcoeffineq2}), as long as 
\[
\left|w\right|
+\frac{\left|a_2\right|}{3}\left|z\right|^3+
\frac{\left|a_3\right|}{4}\left|z\right|^4 +\frac{1}{5}
r^2\sqrt{\frac{2-r^2}{\left(1-r^2\right)^2}-\left|a_2\right|^2r^2-\left|a_3\right|r^4}
\left(\frac{\left|z/r\right|^5}{1-\left|z/r\right|}\right)\leq 0.45.
\]
Here 
$a_3$ must satisfy the restriction
\begin{equation}
\label{restrictiona3}
\left|a_3\right|\leq \frac{1}{r^2}\sqrt{\max\left(\frac{2-r^2}{\left(1-r^2\right)^2}-\left|a_2\right|^2r^2,0\right)}.
\end{equation}
Notice then that, for $a_2$ and $r$ fixed, the expression
\[
\frac{\left|a_2\right|}{3}\left|z\right|^3+
\frac{\left|a_3\right|}{4}\left|z\right|^4 +\frac{1}{5}
r^2\sqrt{\frac{2-r^2}{\left(1-r^2\right)^2}-\left|a_2\right|^2r^2-\left|a_3\right|r^4}
\left(\frac{\left|z/r\right|^5}{1-\left|z/r\right|}\right)
\]
is increasing in $\left|z\right|$. Hence, once we have fixed the radius of the ball in the $z$-plane
(in this case 0.45) as well as the parameter $r$, if we maximize the expression above as a 
function of $\left|a_2\right|$ and $\left|a_3\right|$ subject to restriction (\ref{restrictiona3}) then the difference between 0.45 and this value will give the radius
of a ball in the $w$-plane for which problem (\ref{EarleHamFPPeq}) has a unique solution
(i.e., the radius of the schlicht disk we are after).
Choosing $r=0.8$, and doing as described above, we obtain that $B\geq 0.347493$.
%which is better than Grandjot's estimate and just a bit smaller than Landau's (see \cite{5}).

\section{Can the method be pushed a bit further?}
\label{Pushingthemethod}

As it is now, the estimate given in the previous paragraph seems to be the best estimate that can
be obtained by using fixed point methods. However, we will show 
how these methods might be improved. Let us consider the polynomial
\[
f\left(z\right)=z-\frac{1}{3}z^3-\frac{1}{4}\left(4.66922\right)z^4.
\]
In this case, using the methods described above, we obtain that
$B$ might be taken as the optimal value of
\[
p\left(r\right)=r-\frac{1}{3}r^3-\frac{1}{4}\left(4.66922\right)r^4,
\]
which is obtained at $r\sim 0.534759$, and gives $p\sim 0.38832$, which means that
$f$ covers a schlicht disk of this radius.

We will show directly that this polynomial covers a schlicht disk of radius 
bigger than $0.438$.

First, we let $-1<b<1$, and rewrite the polynomial as its Taylor series centered at
$b$. In the case of the polynomial given above
\[
f\left(z\right)=f\left(b\right)+f'\left(b\right)\left(z-b\right)+F_2\left(z\right),
\]
where
\begin{eqnarray*}
f\left(z\right)&=&b+\frac{1}{3}b^3+\frac{1}{4}\left(4.66922\right)b^4
+\left(1+b^2+\left(4.66922\right)b^3\right)\left(z-b\right)\\
&&\frac{1}{2}\left(2b+3\left(4.66922\right)b^2\right)\left(z-b\right)^2
+\frac{1}{6}\left(2+6\left(4.66922\right)b\right)\left(z-b\right)^3\\
&&+\frac{1}{24}\left(4\times 4.66922\right)\left(z-b\right)^4.
\end{eqnarray*}

The fixed point problem we must solve for this polynomial is
\[
z=\frac{1}{f'\left(b\right)}\left(w-f\left(b\right)-F_2\left(z,b\right)\right)+b,
\]
where $F_2\left(z,b\right)$ has the obvious meaning, and $b$ is also chosen so that $f'\left(b\right)\neq 0$.

We choose then in the $z$-plane a ball centered at $b$ of radius $\rho$, and
we want the right-hand side to map the ball $B_{\rho}\left(b\right)$
into itself. This will happen as long as, for $b$ fixed, $w$
satisfies 
\[
\left|w-f\left(b\right)\right|+
\left|F_2\left(z,b\right)\right|\leq \left|f'\left(b\right)\right|\rho.
\]
Observe then that $F_2\left(z,b\right)$ once $b$ is fixed, being a polynomial, is a holomorphic function
 so its maximum modulus occurs at the circle $\left|z-b\right|=\rho$. Therefore, to obtain
an estimate from below for the radius of the schlicht disk we are looking for, it is
enough to give $b$ and $\rho$ any value and find the maximum modulus 
of $F_2$. Then Bloch constant will be bounded from below by
\[
\left|f'\left(b\right)\right|\rho-\max_{\left|z-b\right|=\rho}\left|F_2\left(z,b\right)\right|.
\]
We have chosen $b=-0.07$ and $\rho=0.59$, which shows
that $f$ covers a schlicht disk of radius $0.43806$.

Interestingly enough, $f$ is not a Bloch function. Indeed, a theorem of Chen-Gauthier (see \cite{Chen1}) guarantees
that the coefficient $a_3$ of a Bloch function must satisfy the estimate
$\left|a_3\right|\leq 4.2$. If we take this into consideration and look at the polynomial
\[
g\left(z\right)=z-\frac{1}{3}z^3-\frac{1}{4}\left(4.2\right)z^4,
\]
then we obtain an estimate from below for the radius of a schlicht disk covered by $g$ of $0.446896$.
%\subsection{On how to improve on known estimates for Bloch constant}

%Here we propose an algorithm, based on the ideas used in the previous section,
%to improve on Bloch constant. If $f$ is a Bloch function, we can estimate
%\[
%f\left(z\right)=z+\sum_{j=2}^m\frac{a_j}{j+1}z^{j+1}+E,
%\]
%and
%\[
%\left|E\right|\leq \frac{1}{m+2}r^2\frac{\sqrt{2-r^2}}{1-r^2}\frac{\left|z/r\right|^{m+2}}{1-\left|z/r\right|}.
%\]
%Therefore, if we consider Bloch polynomials of order at most 28 whose coefficients satisfy
%\[
%\left|a_3\right|\leq 1, \quad 
%\left|a_{j+1}\right|\leq \min_{0<r<1}\frac{1}{r^j}\sqrt{\frac{2-r^2}{\left(1-r^2\right)^2}-\sum_{k=2}^{j}\left|a_k\right|^2r^{2\left(j-1\right)}},
%\]
%and find the Bloch constant of this family, call it $B_{28}$, then an estimate for
%Bloch constant from below would be given by $B_{28}-0.000100508$.
%If we consider polynomials of degree 15, this estimate would be $B_{15}-0.00322629$.

%%%%%%%%%%%%%%%%%%%%%%%%%%%%%%%%%%%%%%%%%%%%%%%
\section{Several complex variables}
\label{Severalvariables}

\subsection{Preliminaries and Notation}
We now apply our methods to the case of several complex variables. But let us firs introduce some notation.

Given a matrix $A$, $\lambda\left(A\right)$ and $\Lambda\left(A\right)$
represent the square root of the minimum and the maximum of the eigenvalues
of $A^*A$. In what follows, $\left\|A\right\|$ represents the operator norm
of $A$. This norm satisfies the following well-known inequalities
\[
\lambda\left(A\right)\leq \left\|A\right\|=\Lambda\left(A\right).
\]
\[
\Lambda\left(A^{-1}\right)=\frac{1}{\lambda\left(A\right)}.
\]
In general, $k=\left(k_1,k_2,\dots,k_m\right)$ will represent a multiindex, and related to a multiindex
we define $\left|k\right|=k_1+k_2+\dots+k_m$ and $k!=k_1!k_2!\dots k_m!$. Regarding differentiation
we write
\[
F^{\left(k\right)}\left(z\right)=\frac{\partial^{\left|k\right|}F}{\partial z_1^{k_1}\partial z_2^{k_2}\dots\partial z_m^{k_m}}.
\]
We will write $z=\left(z_1,z_2,\dots,z_m\right)$, and
obviously $z^{k}=z_1^{k_1}z_2^{k_2}\dots z_m^{k_m}$. For an element $\mathbb{C}^m$, we define
\[
\left|w\right|=\sqrt{\left|w_1\right|^2+\left|w_2\right|^2+\dots+\left|w_m\right|^2},
\]
and 
\[
\left|w\right|_{\infty}=\max_{j=1,\dots,m}\left|w_j\right|.
\]
The ball centered at $\beta$ of radius $r$ will be denoted by $B_r\left(\beta\right)$ and is defined as usual:
\[
B_r\left(\beta\right):=\left\{z\in \mathbb{C}^m:\,\left|z-\beta\right|<r\right\}.
\]
The polydisk centered at $\beta$ of radius $r$, denoted by $D\left(\beta,r\right)$, is defined
as
\[
D\left(\beta,r\right):=\left\{z\in\mathbb{C}^m:\,\left|z-\beta\right|_{\infty}<r\right\},
\]
and its closure will be denoted by $\overline{D}\left(\beta,r\right)$.
\subsection{Wu $K$-mappings.}
We will consider holomorphic maps 
$F:D\left(0,1\right)\subset \mathbb{C}^m\longrightarrow\mathbb{C}^m$,
where $D\left(0,1\right)$ is the unit polydisk,
that satisfy the following estimate
\begin{equation}
\label{Wucondition}
\left\|F'\left(z\right)\right\|\leq K\left|\mbox{det}\left(F'\left(z\right)\right)\right|^{\frac{1}{m}}.
\end{equation}
These maps are called Wu $K$-mappings. As a normalisation
we impose that $\left|\mbox{det}\left(F'\left(0\right)\right)\right|=1$, and we will
assume without loss of generality that $F\left(0\right)=0$ and
that $F$ is bounded. We will show that 
for this family of maps Bloch's Theorem holds.

We begin our proof just as before, by
picking a sequence $\beta_j\in B_1\left(0\right)$, $j=0,1,2,\dots$ as
follows. First, pick $\beta_0=\beta_{\gamma}$
such that $\left|\mbox{det}\left(F'\left(z\right)\right)\right|$ reaches its
maximum at the boundary of the polydisk of radius $\left|\beta_0\right|_{\infty}$ centered at the origin,
at the point $\beta_0$. 
Then choose $\gamma>1$ such that
\[
\left|\beta_{\gamma}\right|_{\infty}\prod_{j=1}^{\infty}\left(1+\gamma^{-j}\right)=1
\]
and a maximal
finite sequence $\epsilon_j$ such that
\[
\frac{\left|\mbox{det}\left(F\left(\beta_{n+1}\right)\right)\right|}
{\left|\mbox{det}\left(F\left(\beta_n\right)\right)\right|}=\gamma^m,
\quad r_{n+1}=r_n\left(1+\epsilon_n\right),
\quad r_j=\left|\beta_j\right|_{\infty},
\]
with $\epsilon_0=0$, $j=0,1,2,\dots,m$,
and $\left|\mbox{det}\left(F'\left(\beta_n\right)\right)\right|$ is the maximum of $\left|\mbox{det}\left(F'\left(z\right)\right)\right|$
in the polydisk of radius $r_n$ centered at the origin.

In this case, the fixed point problem we must solve is 
\[
z=F\left(\beta_n\right)^{-1}\left(w-F\left(\beta_n\right)\right)+\beta_n
-\sum_{\left|k\right|\geq 2}\frac{F\left(\beta_n\right)^{-1}F^{\left(k\right)}\left(\beta_n\right)}
{k!}
\left(z-\beta_n\right)^{k}.
\]
Let us define
\[
g_w:=F\left(\beta_n\right)^{-1}\left(w-F\left(\beta_n\right)\right)+\beta_n
-\sum_{\left|k\right|\geq 2}\frac{F\left(\beta_n\right)^{-1}F^{\left(k\right)}\left(\beta_n\right)}
{k!}
\left(z-\beta_n\right)^{k}.
\]
As in the one dimensional setting, 
a map $G:\Omega\subset \mathbb{C}^m\longrightarrow \mathbb{C}^m$, $\Omega$ convex, which satisfies
that $\left\|G'\left(z\right)\right\|\leq 1-\sigma$, $\sigma>0$, for all $z\in \Omega$, is a contraction.
So again we impose the following condition on $F$, so we make sure $g_w$ is a contraction, 
\[
\sum_{\left|k\right|\geq 2}\frac{\left\|F'\left(\beta_n\right)^{-1}\right\|\left|F^{\left(k\right)}\left(\beta_n\right)\right|}
{\left(k-1\right)!}
\eta^{\left|k\right|-1}\leq 1-\sigma.
\]
Always keep in mind that $k$ is a multiindex.
Now, using Cauchy's estimates, and the fact that for any matrix $A=\left(a_{ij}\right)$
the inequality $\left|a_{ij}\right|\leq \left\|A\right\|$ holds, and that
the ball of radius $\epsilon_{n+1}r_n$ centered at $\beta_n$ is contained in the polydisk
$D\left(0,r_{n+1}\right)$, the previous inequality
can be replaced by the (stronger) condition 
\begin{equation}
\label{contractionconditionscvm}
\sum_{\left|k\right|= 2}^{\infty}
\left|k\right|^m
\frac{\left\|F'\left(\beta_n\right)^{-1}\right\|\max_{z\in\overline{D}\left(0,r_{n+1}\right)}\left\|F'\left(z\right)\right\|}
{\left(\epsilon_{n+1}r_n\right)^{\left|k\right|-1}}
\eta^{\left|k\right|-1}\leq 1-\sigma.
\end{equation}

In what follows, we shall use the notation $\lambda_F\left(z\right)$ and 
$\Lambda_F\left(z\right)$ to indicate $\lambda\left(F'\left(z\right)\right)$
and $\Lambda\left(F'\left(z\right)\right)$ respectively.
Now, we can also estimate, just as before, to make sure that
$g_w:B_{\eta}\left(\beta_n\right)\longrightarrow B_{\eta}\left(\beta_n\right)$, that
\begin{eqnarray*}
\left\|F'\left(\beta_n\right)^{-1}\right\|\left|w-F\left(\beta_n\right)\right|
&\leq& \eta - \sum_{k\geq 2}\frac{\left\|F'\left(\beta_n\right)^{-1}\right\|\max_{z\in\overline{D}\left(0,r_{n+1}\right)}\left\|F'\left(z\right)\right\|}
{k!}
\eta^{\left|k\right|}\\
&\leq& \sigma\eta.
\end{eqnarray*}
Therefore
\[
\left\|w-F\left(\beta_n\right)\right\|\leq \sigma \eta \lambda_F\left(\beta_n\right).
\]
To finish the proof we would need to estimate $\eta$ in terms of $\epsilon_{n+1}$ and $\beta_n$,
and then the argument can proceed as in Section \ref{prooffixedpoint}. The estimate we need for 
$\eta$
is a consequence of (\ref{contractionconditionscvm}), so let us work on this inequality.

First, we estimate
\begin{eqnarray*}
\left\|F´\left(\beta_{n}\right)^{-1}\right\|\max_{\left|z\right|\leq\left|\beta_{n+1}\right|}\left\|F'\left(z\right)\right\|
&=& \frac{\max_{z\in\overline{D}\left(0,r_{n+1}\right)}\Lambda_F\left(z\right)}{\lambda\left(\beta_n\right)}\leq
\frac{K\left|\mbox{det}\left(F'\left(\beta_{n+1}\right)\right)\right|^{\frac{1}{m}}}{\lambda_F\left(\beta_n\right)}\\
&\leq&\frac{K^{m}\left|\mbox{det}\left(F'\left(\beta_{n+1}\right)\right)\right|^{\frac{1}{m}}}
{\left|\mbox{det}\left(F'\left(\beta_{n}\right)\right)\right|^{\frac{1}{m}}}=K^{m}\gamma,
\end{eqnarray*}
where we have used the following fact, which holds for Wu $K$-mappings (and whose proof we leave to the interested reader):
\begin{equation}
\label{Wusmalleigen}
\lambda_{F}\left(\beta_{n}\right)\geq \frac{1}{K^{m-1}}\left|\mbox{det}\left(F'\left(\beta_n\right)\right)\right|^{\frac{1}{m}}.
\end{equation}
So, from (\ref{contractionconditionscvm}), we obtain:
\[
\sum_{\left|k\right|= 2}^{\infty}
\left|k\right|^m
K^{m} \gamma
\left(\dfrac{\eta}{\epsilon_{n+1}\beta_{n}}\right)^{\left|k\right|-1}\leq 1-\sigma.
\]
From the previous inequality we shall produce an inequality for 
$\eta$.
Write $\eta=u\cdot\left(\epsilon_{n+1}\beta_{n}\right)\dfrac{1}{K^m\gamma}$. This yields
\[
K^m\gamma\sum_{\left|k\right|= 2}^{\infty}
\left|k\right|^m
\left(\frac{u}{K^m\gamma}\right)^{\left|k\right|-1}\leq 
\sum_{\left|k\right|= 2}^{\infty}\left|k\right|^m\left(u\right)^{\left|k\right|-1}.
\]
Consider then the function
\[
f\left(u\right):=\sum_{\left|k\right|= 2}^{\infty}\left|k\right|^m\left(u\right)^{\left|k\right|-1}
=\sum_{\left|k\right|=2}^{\infty}\left(\left|k\right|^{\frac{m}{\left|k\right|-1}}u\right)^{\left|k\right|-1}.
\]
Observe that $g\left(x\right)=x^{\frac{1}{x-1}}$, $x\geq 2$, is decreasing function, and its maximum value is 2. Therefore
\[
f\left(u\right)\leq \sum_{\left|k\right|=2}\left(2^m u\right)^{\left|k\right|-1}=\frac{2^mu}{1-2^mu}.
\]
From the inequality
\[
\frac{2^mu}{1-2^mu}\leq 1-\sigma,
\]
we get
\[
u\leq  \frac{1-\sigma}{\left(2-\sigma\right)2^m},
\]
which implies an estimate on $\eta$, which in turn implies the following estimate from below for Bloch's constant
\[
\frac{\sigma\left(1-\sigma\right)}{\left(2-\sigma\right)2^m}\frac{1}{K^m \gamma}\epsilon_{n+1}r_n \lambda_F\left(\beta_n\right).
\]
Using (\ref{Wusmalleigen}), and that by construction
\[\left|\mbox{det}\left(F'\left(\beta_n\right)\right)\right|^{\frac{1}{m}}\epsilon_{n+1}\geq 
\frac{\left|\mbox{det}\left(F'\left(\beta_{\gamma}\right)\right)\right|^{\frac{1}{m}}}{\gamma}\]
this estimate becomes
\begin{equation}
\label{blochwu1}
\frac{\sigma\left(1-\sigma\right)}{\left(2-\sigma\right)2^m}\frac{1}{K^{2m-1} \gamma}
\frac{\left|\mbox{det}\left(F'\left(\beta_{\gamma}\right)\right)\right|^{\frac{1}{m}}}{\gamma}\left|\beta_{\gamma}\right|_{\infty}.
\end{equation}
The inequality above already gives a lower bound for the radius of a schlicht disk
covered by $F$. However, as we did in the one dimensional case, 
we can improve this estimate if we consider the fixed point problem 
\begin{equation}
\label{FPPOSCV}
z=\left(F'\left(0\right)^{-1}\right)w-\sum_{k\geq 2}\frac{F'\left(0\right)^{-1}F^{\left(k\right)}\left(0\right)}{k!}z^{k}.
\end{equation}
In this case, imposing to the left hand side of (\ref{FPPOSCV}) to be a contraction, using Cauchy's estimates, 
and Wu's condition, we obtain that
\[
\sum_{\left|k\right|= 2}^{\infty}
\left|k\right|^m
K \left|\mbox{det}\left(F'\left(\beta_{\gamma}\right)\right)\right|^{\frac{1}{m}}
\left(\dfrac{\eta}{\left|\beta_{\gamma}\right|_{\infty}}\right)^{\left|k\right|-1}\leq 1-\sigma.
\]
Proceeding as before, we obtain the estimate
\[
\frac{\sigma\left(1-\sigma\right)}{\left(2-\sigma\right)2^m}\frac{1}{K \left|\mbox{det}\left(F'\left(\beta_{\gamma}\right)\right)\right|^{\frac{1}{m}}}
\left|\beta_{\gamma}\right|_{\infty}\lambda_{F}\left(0\right)
\]
which gives an estimate from below
\begin{equation}
\label{blochwu2}
\frac{\sigma\left(1-\sigma\right)}{\left(2-\sigma\right)2^m}\frac{1}{K^{m} \left|\mbox{det}\left(F'\left(\beta_{\gamma}\right)\right)\right|^{\frac{1}{m}}}
\left|\beta_{\gamma}\right|_{\infty}
\end{equation}
Now, notice that (\ref{blochwu1}) is better than (\ref{blochwu2}) when
\[
\left|\mbox{det}\left(F'\left(\beta_{\gamma}\right)\right)\right|^{\frac{1}{m}}\geq \gamma K^{\frac{m}{2}-\frac{1}{2}},
\]
whereas (\ref{blochwu2}) is better than (\ref{blochwu1}) when the opposite inequality holds.
From these estimates we conclude that the Bloch constant for Wu $K$-mappings is bounded from below
by
\[
\frac{\sigma\left(1-\sigma\right)}{\left(2-\sigma\right)2^m}\frac{1}{K^{\frac{3}{2}m-\frac{1}{2}}}
\frac{\left|\beta_{\gamma}\right|_{\infty}}{\gamma}
\geq \frac{\sigma\left(1-\sigma\right)}{\left(2-\sigma\right)2^m}\frac{1}{K^{\frac{3}{2}m-\frac{1}{2}}}
\frac{e^{-\frac{1}{\gamma-1}+\frac{1}{2}\frac{1}{\gamma^2-1}-\frac{1}{3}\frac{1}{\gamma^3-1}}}{\gamma}.
\]
Let us write the result whose proof we have given above. 
\begin{theorem}
\label{Blochpolyscv}
Let $F:D\left(0,1\right)\subset \mathbb{C}^m\longrightarrow\mathbb{C}^m$ be a Wu $K$-mapping, normalised so that
$\left|\mbox{det}\left(F'\left(0\right)\right)\right|=1$. Then there
is a constant $C\left(m\right)>0$, which only depends on the dimension, such that
$F$ covers a schlicht disk of radius at least $C\left(m\right)/K^{\frac{3}{2}m-\frac{1}{2}}$.
\end{theorem}

To obtain an estimate for Bloch's constant for maps defined in the unit ball, just notice that 
$D\left(0,1/\sqrt{m}\right)\subset B_{1}\left(0\right)$; in this way, we obtain a similar estimate
for these maps as the one given in Theorem \ref{Blochpolyscv}, we just have to multiply the estimate by
$m^{-\frac{1}{2}}$.
Notice that the 
best estimates obtained before the paper  \cite{Chen2}, where an estimate from below $\sim\dfrac{1}{K^{m-1}}$
for Bloch's constant of Wu $K$-mappings is shown, gave an
estimate $\sim \dfrac{1}{K^{2m}}$ for the asymptotic behavior of Bloch's constant (see \cite{Chen2} and \cite{Harris}).

\bibliographystyle{amsplain}

\end{document}